\documentclass[12pt]{article}
\usepackage[margin=1.35in]{geometry}
\usepackage{amsmath,yhmath}
\usepackage{amsthm}
\usepackage{amssymb}
\usepackage{amsfonts}
\usepackage{bbm} 
\usepackage{cite}
\usepackage{tikz}
\usepackage{enumitem}
\usepackage{subfig}
\usepackage{caption}

\usepackage{microtype}
\usepackage{lmodern}

\usepackage{tikz}
\usetikzlibrary{shapes,arrows,graphs,graphs.standard,shapes.misc}
\usetikzlibrary{matrix,decorations.pathreplacing, calc, positioning,fit}
\usetikzlibrary{shapes.geometric}
\usetikzlibrary{cd}

\newcommand{\mab}[1]{\vspace{.1cm}

\noindent {\bf #1} }
\newtheorem{definition}{Definition}

\newtheorem{theorem}{Theorem}
\newtheorem{lemma}{Lemma}

\newtheorem{proposition}[theorem]{Proposition}

\makeatletter
\xdef\@endgadget#1{{\unskip\nobreak\hfil\penalty50\hskip1em\hbox{}\nobreak\hfil#1\parfillskip=0pt\finalhyphendemerits=0\par}}
\newcommand\@Endofsymbol{$\triangledown$}
\newcommand\Endofremark{\@endgadget{\@Endofsymbol}}
\makeatother

\newcommand{\R}{\mathbb{R}}

\newcommand{\cL}{\mathcal{L}}

\newcommand{\ddt}{\frac{d}{dt}}
\newcommand{\ddtk}{\frac{d^{k}}{dt^{k}}}
\newcommand{\ddtkp}{\frac{d^{k+1}}{dt^{k+1}}}

\renewcommand{\i}{$(i)$ }
\newcommand{\ii}{$(ii)$ }
\newcommand{\iii}{$(iii)$ }

\DeclareMathOperator{\ad}{ad}

\newcommand{\Span}{\operatorname{span}}

\newcounter{cAss}
\newcounter{cAssSaved}
\newcommand\Ass[1]{\ensuremath{\boldsymbol{\mathcal A}_{\text{\hspace{0.75pt}\bf#1}}}}
\newlength\asswidth
 {\begin{list}{\Ass{\arabic{cAss}}:}{%
   \setcounter{cAssSaved}{\thecAss}\usecounter{cAss}\setcounter{cAss}{\thecAssSaved}
   \setlength\topsep{.667ex plus 2pt minus 2pt}
   \setlength\itemsep\topsep
   \settowidth\asswidth{\Ass{\arabic{cAss}}:}
   \addtolength\asswidth\labelsep
   \setlength\leftmargin\asswidth
   \setlength\labelwidth\asswidth}}%
 {\end{list}}

\usepackage[color=yellow, textsize=small, textwidth=\marginparwidth, draft]{todonotes}   

\usepackage{xcolor}
\definecolor{forestgreen}{rgb}{0.13, 0.55, 0.13}
\definecolor{orange}{rgb}{1,0.49,0}

\newcounter{para}[section]

\graphicspath{{Figures/}}

\begin{document}

\title{Canonical forms for polynomial systems with balanced super-linearizations}
\author{Mohamed-Ali Belabbas
}
\date{}\maketitle
\begin{abstract}
A system is Koopman super-linearizable if it admits a finite-dimensional embedding as a linear system. Super-linearization is used  to leverage methods from linear systems theory to design controllers or observers for nonlinear systems. We call a super-linearization balanced if the degrees of the hidden observables do not exceed the ones of the visible observables. We show that systems admitting such super-linearization can be put in a simple canonical form via a linear change of variables.\end{abstract}
\fussy

\section{Introduction}
In the early 1930s,~\cite{carleman1932application} and ~\cite{koopman1931hamiltonian,kowalski1991nonlinear} introduced the idea of linearizing a system dynamics by embedding it in an infinite-dimensional vector space. The additional degrees of freedom allow to linearize the system without recourse to changes of variables. These embeddings are still actively studied a century later. A problem that remains open is whether such an embedding into a finite-dimensional space exists; this is called a super-linearization. Said otherwise, can a system dynamics be linearized by the addition of a finite set of observables? In this paper, we address this problem by deriving canonical forms for super-linearizable systems with polynomial vector fields.

 The notion of embedding of systems as linear systems has found early applications in nonlinear control~\cite{brockett1976volterra, brockett2014early}, and there has been a renewed interest spurred in part by data-driven methods in control~\cite{mauroy2020koopman, otto2021koopman}. Other approaches to linearize control system dynamics  rely on utilizing transformation groups acting on the system without recourse to embeddings in higher dimensional state-spaces, viz, via diffeomorphisms or feedback and diffeomorphisms---the so-called feedback linearization. In this vein, we cite the seminal works of Brockett~\cite{brockett1978feedback}, as well as Isidori and Krener~\cite{isidori1982feedback} and Jacubzyck and Respondek~\cite{jacubczyk1980linearization}.
 
  We consider here the controlled system 
\begin{equation}\label{eq:mainsys}
\dot x = f(x)+ug(x),	
\end{equation}
where $x \in \R^n$, $f,g$ are smooth vector fields in $\R^n$.  This system admits a linear finite-dimensional embedding, or {\em super-linearization} or finite-dimensional Koopman linearization, if there exists $m\geq 0$ functions, called  observables, which when adjoined to the original system would permit its linearization. A typical example~\cite{brunton2016koopman} is the following two-dimensional system
\begin{equation}\label{eq:ex1}
\begin{cases}
\dot x &= -x+y^2 +u \\
\dot y &= -y	
\end{cases}
\end{equation}
Adding the observable $w:=y^2$, whose total time derivative is given by $\dot w = 2y\dot y = -2y^2 =-2w$, we obtain the three-dimensional {\em linear} system
\begin{equation}\label{eq:ex2}
\begin{cases}
\dot x &= -x+w +u \\
\dot y &= -y\\
\dot w &=-2w.	
\end{cases}
\end{equation}
Define the projection map $\Pi(x,y,w):=(x,y)$. We call $\Pi$ a {\em standard projection} and~\eqref{eq:ex2} a linear {\em finite-dimensional embedding} with {\em observable} $p:(x,y) \mapsto y^2$. We precisely define these notions below.  We see that solutions of~\eqref{eq:ex2} with initial conditions $(x_0,y_0,y_0^2)$ are mapped by $\Pi$ to solutions of~\eqref{eq:ex1} with initial conditions $(x_0,y_0)$. 

We make the following informal observations about this example:

\begin{enumerate}
\item The nonlinear part of the vector field {\em depends solely} on the $y$-variable.
\item The nonlinear part of the vector field {\em does not} affect the $y$-variable.
\item The control term in~\eqref{eq:ex1} does not affect the $y$-variable.
\end{enumerate}
We see that the dynamics of $y$ is linear and autonomous; in fact~\eqref{eq:ex1} is equivalent to the time-varying linear system

$$\dot x = -x+e^{-2t}y_0^2 +u.$$

In this paper, we provide a canonical form for  systems as the one in~\eqref{eq:ex1}  which admit a finite-dimensional Koopman linearization.  More precisely, we identify a class of systems for which  the observations just made   always hold true, and leave the study of the general case to a subsequent publication. In order to describe the class of systems, we rely on the notions of {\em visible} and of {\em hidden} observables introduced in~\cite{belabbasobs2022}. Roughly speaking, an observable is visible if it appears explicitly in the nonlinear dynamics that we seek to linearize, and it is hidden otherwise. A precise definition is given below. In the example~\eqref{eq:ex1}, the only observable is $y^2$, and it is thus visible. The class of systems we focus on is the class of polynomial systems admitting what we call a {\em balanced} super-linearization (see Definition~\ref{def:indepobs} below). For clarity of exposition, we also focus on the case of systems admitting a super-linearization with one visible observable, but with  an arbitrary number of hidden observables. The ideas developed here to obtain the canonical form  can however be applied to more general cases.

\section{Preliminaries} Throughout the paper
$$
(f,g) \mbox{ denotes the system }\dot x = f + u g.
$$ 
 We let $e^{t(f+ug)}x_0$ be the solution at time $t$ of~\eqref{eq:mainsys} with initial state $x_0$ and control $u$. 
Given positive integers $n$ and $m$, we denote by $\Pi:\R^{n+m} \to \R^n$ the standard projection onto the first $n$ variables and $\bar \Pi:\R^{n+m} \to \R^m$ the standard projection onto the last $m$ variables, i.e., 
\begin{equation*}
	\Pi(z_1,\ldots,z_{n+m})=(z_1,\ldots,z_n)\mbox{ and }\bar \Pi(z_1,\ldots,z_{n+m})=(z_{n+1},\ldots,z_{n+m}).
\end{equation*} 
The integers $n$ and $m$ will be clear from the context. 
With some abuse of notation, for $z \in \R^{n + m}$ we also set  $z_1:= \Pi(z) \in \R^n$ and $z_2:= \bar \Pi(z) \in \R^m$, so that $z= \begin{bmatrix} z_1^\top & z_2^\top \end{bmatrix}^\top .$ In order to simplify the notation, we also write $z=\begin{pmatrix} z_1 & z_2 \end{pmatrix}$, with the understanding that $z, z_1$ and $z_2$ are column vectors.
Given a map $p:\R^n \to \R^m$, we say that $p$ has a constant term if $p(0) \neq 0$ and that $p$ has a linear term if $\frac{d}{dx}|_{x=0}p(x) \neq 0$. 
 
We let
$$\iota(x):=\begin{pmatrix} x & p(x)\end{pmatrix},$$
where the function $p$ will be clear from context.  
Given a smooth map $\Pi$ between smooth manifolds, we denote by $d\Pi$ its Jacobian. Given smooth vector fields $f,g$, we denote the Lie derivative~\cite{isidori1985nonlinear} of $g$ along $f$ by $\ad_f g$, i.e.,
$$\ad_fg := [f,g] =\frac{\partial g}{\partial x} f-\frac{\partial f}{\partial x}g.$$ For $p:M \to \R^m$ a smooth (vector-valued) function, we define its directional derivative along $f$, also called Lie derivative,  as
$$\cL_f p = f \cdot p = \frac{\partial p}{\partial x} f.
$$
We recall the well-known identity
\begin{equation}\label{eq:propliebracket}
\cL_{[f,g]} p = f\cdot g \cdot p-g \cdot f \cdot p.	
\end{equation}

We let $I$ be the identity matrix whose dimension is determined by the context or explicitly indicated via an index. For a polynomial map $p:\R^n \to \R^m$, we denote by $\deg p$ its degree and $p^{[i]}$ its {\em homogeneous part} of degree $i$, so that 
$$p(x) = \sum_{i=0}^{\deg p} p^{[i]}(x). $$ If $i <0$, then $p^{[i]}=0$.

The following fact will be used repeatedly in the proofs, without making explicit mention of it: if $f$ is a polynomial vector field with terms of least degree $d_m$ and highest degree $d_M$ and $p$ a real-valued polynomial with terms of least degree $d_m'$ and highest degree $d_M'$ , then the non-zero terms of $\ad_fp$  are  of degree at least $d_m+d_m'-1$ and at most $d_M+d_M'-1$.  In particular, if $f$ is an affine polynomial, then the non-zero terms in $\ad_fp$ have degrees at least $d_m'$ and at most $d_M'$, similarly to $p$. 

An {\em affine control system} is a controlled differential equation of the form
\begin{equation}\label{eq:canonicalmainsys}\dot z = A_\ell z +B_\ell u + D_\ell,
\end{equation}
for a matrix $A_\ell$ and vectors $B_\ell,D_\ell$ of the appropriate dimensions for the equation to be well-defined.
We refer to the affine control system above as {\em the triple} $(A_\ell, B_\ell, D_\ell)$.

We denote by $C(A,B)$ and $O(A,G)$ the controllability and observability matrices, respectively, associated with the system
\begin{equation}\label{eq:assoclinsys}
\begin{cases}\dot x &= Ax+Bu\\  y&=Gx.
\end{cases}
\end{equation}

\section{Statement of the main results} 

We first recall the definition of super-linearization. 

\begin{definition}[Super-linearization]\label{def:pifeedbackequiv}We say that the system $(f, g)$ in $\R^n$ is super-linearized to $(A_\ell,B_\ell, D_\ell)$  with $A_\ell \in \R^{(m+n)\times(m+n)}$ and $B_\ell, D_\ell \in \R^{m+n}$ if there exists a smooth map $p:\R^n \to \R^{m}$   so that for all $x_0 \in \R^n$, it holds that
\begin{equation}\label{eq:equivalence1}
\Pi\left(e^{t(A_\ell z+B_\ell u +D_\ell)}z_0 \right) = e^{t(f+ug)}x_0	\mbox{ with } z_0=\iota(x_0)
\end{equation} 
We call the functions $p_j:\R^n \to \R$, $j=1,\ldots,m$,  the {\em observables} and the data of $\cL:=(A_\ell, B_\ell, D_\ell,p)$ an  finite-dimensional  embedding as an affine system. 
\end{definition}
\noindent The definition can be summarized in the commutative diagram below 
\begin{center}
\begin{tikzcd}[column sep=huge, row sep=huge]
  \R^n \arrow[r, "e^{t(f+ug)}"] \arrow[d,  "(I \,\ p)" left ] & \R^{n} \\
  \R^{n+m} \arrow[r,  "e^{t(A_\ell z+B_\ell u+D_\ell)}" below] & \R^{n+m} \arrow[u,  "\Pi" right]
  \end{tikzcd}
\end{center}

In order to state the main result of this paper, we recall some concepts and results from~\cite{belabbasobs2022} that are needed. Given a super-linearization of a system $(f,g)$ to $(A_\ell, B_\ell, D_\ell,p)$ with $m$ observables, we partition $A_\ell, B_\ell, D_\ell$ as
\begin{equation}\label{eq:partabell}
A_\ell = \begin{bmatrix} A & G \\ H & M \end{bmatrix}, B_\ell = \begin{bmatrix} B \\ C
\end{bmatrix}\mbox{ and } 
D_\ell = \begin{bmatrix}
D \\ E
\end{bmatrix}
\end{equation}
where $A \in \R^{n \times n}, G\in \R^{n \times m}, H\in \R^{m \times n}$ and $M\in \R^{m \times m}$ partition $A_\ell$,  $B \in \R^{n}, C \in \R^m$  partition $B_\ell$ and $D \in \R^{n}, E \in \R^m$ partition of $D_\ell$.

We start with the following statement, which is found in~\cite[Theorem 1, Corollary 2]{belabbasobs2022}

\sloppy
\begin{theorem}[Th.~1, Cor.~2 in~\cite{belabbasobs2022}]\label{th:recmain1cor2}
Assume that the system $(f,g)$ is super-linearized to $(A_\ell,B_\ell, D_\ell)$ with observables $p:\R^n \to \R^m$. Then it holds that
\begin{enumerate}
\item The system $(f,g)$  is of the form
\begin{equation}\label{eq:canon1}\dot x = A x +G p + B u+D.\end{equation}
 \item The observables 	$p$ satisfy the pair of partial differential equations
\begin{equation}\label{eq:master1}
\left\lbrace\begin{aligned}&	G\frac{\partial p}{\partial x}(Ax+Gp(x)+D)=G\left(Hx+Mp(x)+E\right) \\
&G\frac{\partial p}{\partial x}B=GC
\end{aligned}
\right.	
\end{equation}
\item  Denoting by  $x(t)$ the solution of~\eqref{eq:mainsys}, by $z(t)$ is the solution of~\eqref{eq:canonicalmainsys} with $z(0)=\iota(x_0)$ and letting $z_2(t)=\bar \Pi z(t)$, we have 
\begin{equation}\label{eq:GpGz}
Gp(x(t))=Gz_2(t).
\end{equation}

\end{enumerate}

\end{theorem}
\fussy
Equation~\eqref{eq:GpGz} underscores the importance of the submatrix $G$ in the partition of $A_\ell$ given in~\eqref{eq:partabell}. We will refer to $G$ as the {\em $G$-matrix} of the super-linearization. It allows us, in particular, to define visible and hidden observables of a super-linearization (see~\cite{belabbasobs2022}).
 \sloppy
\begin{definition}[Visible and hidden observables]\label{def:vishidobs}
Let the triplet $(A_\ell, B_\ell, D_\ell)$ be a super-linearization of $(f,g)$ via $p:\R^n \to \R^m$ with $A_\ell$ partitioned as in~\eqref{eq:partabell}. Then  $p_j:\R^n \to \R$, $j=1,\ldots,m$, is a {\em visible observable} for the super-linearization if there exists $i \in \{1,\ldots, n\}$   so that $G_{ij} \neq 0$. Otherwise, $p_j$ is a {\em hidden observable}.	
\end{definition}
\fussy
Super-linearizations for systems, when they exist, are not unique and their number of observables can vary from one super-linearization to another. In~\cite{belabbasobs2022}, we exhibited a system invariant, namely the rank of $G$-matrices of super-linearizations in the so-called {\em reduced visible form.} More precisely, we call a super-linearization of $(f,g)$ to $(A_\ell,B_\ell,D_\ell)$ via $p$    in {\em reduced visible form}  if $p$ has no linear nor constant terms, and the visible observables  are {\em linearly independent}. 
We recall that the entries of $p:\R^n \to \R^m$ are said to be {\em linearly independent} if  for all $w \in \R^m$ 
$$w^\top p=0 \Leftrightarrow w=0.$$
We showed that systems admitting a super-linearization necessarily admit a super-linearization in reduced visible form, and that the ranks of the $G$-matrices of any such super-linearization are the same.
\sloppy
\begin{theorem}[Th.~3, Prop.~9 in~\cite{belabbasobs2022}]\label{th:rvf}
Assume  that $(f,g)$ admits a super-linearization. Then, it admits a super-linearization in {\em  reduced visible form}. Furthermore, all super-linearizations in reduced visible forms have $G$-matrices of the same rank.
\end{theorem}
\fussy
We are now in a position to describe what is meant by {\em balanced} super-linearization of a polynomial system. 

\begin{definition}[Balanced super-linearization]\label{def:indepobs}
Consider the polynomial $(f,g)$. We say that it admits a {\em balanced} super-linearization $(A_\ell,B_\ell,D_\ell,p)$ if the degree of the hidden observables does not exceed the degree of the visible observables.
\end{definition}

The system~\eqref{eq:ex1} has a balanced super-linearization, since it does not have any  hidden observables.
In the sequel, we always assume that the system $(f,g)$ is super-linearized to $\cL$, where $\cL$ is balanced and in reduced visible form. The latter can be assumed, owing to Theorem~\ref{th:rvf}, without loss of generality. Putting this together with Theorem~\ref{th:recmain1cor2}, we can assume that if a polynomial system has a super-linearization, it has  a super-linearization with observables of degrees two or higher. The main result then states that if a polynomial system has a balanced super-linearization with a single visible observable, then there exists  a linear change of variables so that we can partition the state-variable in the basis, say $x'$, into $x_1'$ and $x_2'$, and the system in these variables is of the form

\begin{equation}\label{eq:canonicalformexplicit}
\frac{d}{dt} \begin{bmatrix}
 x'_1\\x'_2
 \end{bmatrix} = \begin{bmatrix} A'_{11} & A'_{12}\\
 0 &  A'_{22} 	
 \end{bmatrix}\begin{bmatrix}
 x'_1\\x'_2
 \end{bmatrix}+\begin{bmatrix}
p'(x_2)\\
0	
\end{bmatrix}
+u\begin{bmatrix}
 B'\\
0	
\end{bmatrix}+D'.
\end{equation}
Observe that only the dynamics of $x'_1$ is nonlinear, and that the nonlinearity solely depends on $x_2'$.
To simplify the notation, we will also assume  that the origin is an equilibrium, i.e., $f(0)=0$, and $D=0$. More precisely, we have

\begin{theorem}\label{th:canon}
Assume that the polynomial system $(f,g)$, with $f(0)=0$,  admits  a {\em balanced} super-linearization with one {\em visible } observable. Then, there exists a linear change of variables $x'=Px$ and a partition $x'=(x'_1\, \,x'_2)$ of $x'$ so that the nonlinear term of the dynamics~\eqref{eq:mainsys} expressed in $x'$ variables is a polynomial $p'(x'_2)$ and the dynamics of $x'_2$ is \emph{linear and autonomous}.
\end{theorem}
The canonical form extends to systems with more than one visible observables, but we focus here on the single observable case for clarity of exposition. 

\section{Proof of the main results}
We now prove Theorem~\ref{th:canon}.

Recall that from Theorem~\ref{th:rvf}, we can assume without loss of generality that if a system is super-linearized via a polynomial map $p :\R^n \to \R^m$, then it can be super-linearized with polynomial observables that do not have constant nor linear terms. This fact plays an important role below and we always assume from now on that the observables are polynomial maps of degree $d \geq 2$.

Furthermore, we assume that the $G$-matrix of the super-linearization is of  rank $m_v=1$. This can  be assumed without loss of generality using Theorem~\ref{th:rvf}.	We denote by $\cL$  such a super-linearization, and we let $q:\R^n \to \R$ be the visible observable, and let $p_{2},\ldots,p_m$ be the hidden observables. Hence, the $G$-matrix can be expressed as $\bar G e_1^\top$, where $\bar G \in \R^n$; recall that $e_i$ is the column vector with all entries zero save for the $i$th entry, which is one.
	
Theorem~\ref{th:canon} will be a consequence of the next Proposition. 

\begin{proposition}\label{prop:mainprop}
Under the assumptions of Theorem~\ref{th:canon}, let $q:\R^n \to \R:x \mapsto e_1^\top p$ be the visible observable and let the $G$-matrix be $G= \bar G e_1^\top$.  Then, for $1 \leq i \leq m_v$, 
\begin{equation}\label{eq:dpsiconstantB}
\frac{\partial q}{\partial x}v =0 \mbox{ for all }  v \in C(A,B), \mbox{ and all } x \in \R^n 
\end{equation}
and 
\begin{equation}\label{eq:dpsiconstantA12}
\frac{\partial q}{\partial x} v =0 \mbox{ for all } v \in O(A,G)\mbox{ and all } x \in \R^n	
\end{equation}
\end{proposition}
We recall that $C(A,B)$ and $O(A,G)$ are the controllability and observability matrices of the linear system in~\eqref{eq:assoclinsys} and $A$ and $B$ are as in the partition of $A_\ell, B_\ell$ described in~\eqref{eq:partabell}.

 We need the following Lemma for the proof of Proposition~\ref{prop:mainprop}.

\begin{lemma}\label{lem:dpsiA12}
	Let $\psi:\R^n \to \R$  be a polynomial of degree  $d$ and with $\psi^{[0]}=\psi^{[1]}=0$,  let $q:\R^n \to \R$ be a polynomial of  degree $d'$, let   $\varphi:\R^n \to \R$ a polynomial  of degree at most $d'$, and $\bar G \in \R^n$. If
	\begin{equation}\label{eq:Gqphi}
\cL_{\bar Gq} \psi = \varphi		
	\end{equation}
 then
	$$ \cL_{\bar Gq}\psi=\cL_{\bar G}\psi=0.$$    
\end{lemma}

\begin{proof} If $d <2$, there is nothing to prove. We thus assume that $d \geq 2$. If $d'=0$, then $q$ is a constant and the result holds from the linearity of the Lie derivative. We thus assume that $d' \geq 1$.

Recall that we can write $q$ as $\sum_{i=0}^{d'} q^{[i]}$, where $q^{[i]}:\R^n \to \R $ is a homogeneous polynomial of degree $i$.  The proof is by induction on the degree of the homogeneous terms in $\cL_{\bar Gq}\psi$. The right-hand-side of $\cL_{\bar Gq} \psi = \varphi$ has terms of degree at {\em most} $d'$, whereas the left-hand-side has terms of degrees up to $d+d'-1$. We equate the terms of the same degree in $x$, for degrees $d'+1$ to $d+d'-1$. Starting with degree $d+d'-1$, we have
	\begin{equation}\label{eq:eqiter2d-1} \frac{\partial \psi^{[d]}}{\partial x} \bar G q^{[d']}=0 \mbox{ for all } x \in \R^n.
	\end{equation}
	Now let  $Z(q) \subseteq \R^n$ be the zero-set of $q$,  it holds that 
	\begin{equation}\label{eq:eqiterc1}\frac{\partial \psi^{[d]}}{\partial x} \bar G =0 \mbox{ for }  x \notin Z(q^{[d']}).
	\end{equation}
	Because $q^{[d']}$ is a non-zero polynomial, $Z(q^{[d']})$ is closed and of Lebesgue measure zero in $\R^n$ (see, e.g.,~\cite{lojasiewicz1964triangulation}). The function $\frac{\partial \psi^{[d]}}{\partial x} \bar G$ is thus a continuous function, defined on $\R^n$, which vanishes identically  on an open dense subset of $\R^n$. We conclude that it is zero everywhere on $\R^n$, i.e.,  
		\begin{equation}\label{eq:iterc1p}\frac{\partial \psi^{[d]}}{\partial x} \bar G \equiv 0.
	\end{equation}

	Equating the terms of degree $d+d'-2$ in~\eqref{eq:Gqphi}, we have
	\begin{equation}\label{eq:eqiter2d-2}  
	\frac{\partial \psi^{[d]}}{\partial x} \bar G q^{[d'-1]}+\frac{\partial \psi^{[d-1]}}{\partial x} \bar G q^{[d']}=0 \mbox{ for all } x \in \R^n.
	\end{equation} 
	
	Using~\eqref{eq:iterc1p}, the above equation simplifies to
	$$
	\frac{\partial \psi^{[d-1]}}{\partial x} \bar G q^{[d']}=0 \mbox{ for all } x \in \R^n,
	$$ from which we obtain, using the same argument as above involving the zero set of $q^{[d']}$, that 
	\begin{equation}\label{eq:eqiterc2}
	\frac{\partial \psi^{[d-1]}}{\partial x} \bar G=0.
	\end{equation}

	For the general case, let $1 \leq k \leq d+d'-2$ and assume that
	\begin{equation}\label{eq:eqiterpc3}
	\frac{\partial \psi^{[d-j]}}{\partial x} \bar G=0 \mbox{ for all } x \in \R^n \mbox{ and all } 0 \leq j \leq k-1.
	\end{equation}
	We show that $\frac{\partial \psi^{[d-k]}}{\partial x}\bar G=0$ for all $x \in \R^n$. Equating the terms of degree $d+d'-k-1$ in~\eqref{eq:Gqphi},  we have
	\begin{equation}\label{eq:eqiterpck}
	\sum_{j=0}^{k} \frac{\partial \psi^{[d-j]}}{\partial x} \bar G q^{[d'-k+j]}=0 \mbox{ for all } x \in \R^n.
	\end{equation}
From the induction hypothesis, we have that all terms in the summation above vanish save for the term corresponding to $j=k$. Hence,~\eqref{eq:eqiterpck}  reduces to $$ \frac{\partial \psi^{[d-k]}}{\partial x} \bar G q^{[d']}=0 \mbox{ for all } x \in \R^n.
	$$
	Using the same argument as above, we conclude that 
$$\frac{\partial \psi^{[d-k]}}{\partial x}\bar G =0 \mbox{ for all } x \in \R^n$$ 
as announced. Hence $\frac{\partial \psi^{[k]}}{\partial x} \bar G=0$ for $0 \leq k \leq d$, which implies that $\frac{\partial \psi}{\partial x} \bar G  =\cL_{\bar G} \psi=0$.  This concludes the proof.
 \end{proof}

We now prove  Proposition~\ref{prop:mainprop}.

\begin{proof}[Proof of Proposition~\ref{prop:mainprop}]
The proof will be by induction over $k$ for the following 3 items:

\begin{enumerate}
\item[\i] $\cL_{A^j\bar G}\cL_{Ax}^{k-j-1}	q=0 $ for all $x \in \R^n$ and $0 \leq j \leq k-1$.
\item[\ii] $\cL_{A^jB}\cL_{Ax}^{k-j-1}	q=0 $ for all $x \in \R^n$ and $0 \leq j \leq k-1$.
\item[\iii] $\frac{d^{k}}{dt^{k}} Gp(x(t))=GM^{k}z_2= \cL^k_{Ax} Gp$ 
\end{enumerate}
Because $G=\bar G e_1^\top$ with non-zero $\bar G$, the third item can be expressed equivalent as
$$
\frac{d^{k}}{dt^{k}} q(x(t))=GM^{k}z_2= \cL^k_{Ax} q.
$$
Though all the items above are needed to perform the induction, the statement of Proposition~\ref{prop:mainprop} is comprised of  items~\i and~\ii for $j=k-1$.

\mab{Base case $k=1$.} 
On the one hand, using~\eqref{eq:GpGz} and the expression for the super-linearized dynamics, we have
\begin{align}\label{eq:ddtgp0}
\ddt Gp(x(t))& = G\ddt z_2\notag \\
&= G(Hz_1+Mz_2+Cu+E).	
\end{align}
  On the other hand, using the chain rule and the standard form of the dynamics given in Theorem~\ref{th:recmain1cor2},  we have
 \begin{equation}\label{eq:ddtgp1}
\frac{d}{dt} Gp(x(t)) = G \frac{\partial p}{\partial x}(Ax+Gp + Bu)	
 \end{equation}
Evaluating~\eqref{eq:ddtgp0} and~\eqref{eq:ddtgp1} at $t=0$, where we recall that the initial state for $z$ is then given by $\iota(x) = \begin{pmatrix} x & p(x) \end{pmatrix}$, we obtain the relation

\begin{equation}\label{eq:equatin1}
\frac{d}{dt}|_{t=0} Gp(x(t))=G\frac{\partial p}{\partial x} (Ax+Gp(x)+Bu) = G(Hx+Mp(x)+Cu+E)	
\end{equation}
Since $p$ has terms of degree at least 2, we have that  $\frac{\partial p}{\partial x}$ is a matrix of polynomials with degrees at least 1 and we conclude by matching the terms of the same degree in $x$ and $u$ that 
\begin{equation}\label{eq:somezeros}
G\frac{\partial p}{\partial x}B =0 \mbox{ and } GC=GE= GH=0.
\end{equation}	
Using these relations in~\eqref{eq:ddtgp1}, we have 
\begin{equation}\label{eq:ddtgp3}
\frac{d}{dt} Gp(x(t))=\cL_{Ax+Gp}Gp,
\end{equation}
and using them in~\eqref{eq:ddtgp0}, we get
\begin{equation}\label{eq:ddtgp4}
\frac{d}{dt} Gp(x(t))=GMz_2.
\end{equation}
This proves the {\em first part of item~3}.

Equating the right-hand sides of~\eqref{eq:ddtgp3} and~\eqref{eq:ddtgp4} at $t=0$ and using $G=\bar G e_1^\top$, we obtain 

\begin{equation}\label{eq:masterc}
\bar G \frac{\partial q}{\partial x}(Ax+\bar G q(x))=\bar Ge_1^\top Mp(x) 
\end{equation}
From that relation  together with fact that $\bar G$ is nonzero,  we have 
$$
\frac{\partial q}{\partial x}\bar Gq = e_1^\top Mp-\frac{\partial q}{\partial x} Ax=:\varphi(x)
$$

Because the super-linearization is balanced,  the polynomial $\varphi(x)$ has degree at most equal to the degree of $q$. Hence, by Lemma~\ref{lem:dpsiA12}, it holds that $\cL_{\bar G}q=0$. This proves  {\em item~1} for the case $k=1$.

For the {\em second item}, we have from~\eqref{eq:somezeros} that $$
G\frac{\partial p}{\partial x}B = \bar G \frac{\partial q}{\partial x} B=0
$$ and, since $\bar G \neq 0$, that $\frac{\partial q}{\partial x}B=0$. 

For the {\em second part of item~3},  we get from~\eqref{eq:ddtgp1}, and the fact that $G=\bar G e_1^\top$, the relation
$$
\frac{d}{dt} q =  \frac{\partial q}{\partial x} (Ax+\bar G q+Bu).
$$

We have proven above that $\cL_{ \bar G}q=0$  and $\cL_B q = 0$. Plugging these relations,  we get that $$\ddt q = \frac{\partial q}{\partial x}Ax = \cL_{Ax} q.$$ Since $Gp=\bar G q$, multiplying the  above equation by $\bar G$ on the left yields the second equation of item~\iii for $k=1$. This concludes the proof of the base case.

\mab{Inductive step for $k > 1$.} 
    
We assume that items 1 through 3 hold for $1,\ldots,k$, and show that they hold for $k+1$. We start by evaluating the $(k+1)$th  time derivative of $Gp$. Using item~3 of the induction hypothesis, we have  on the one hand that 
\begin{align}\label{eq:ddkgp1}
\ddtkp Gp(x) &= \ddt \ddtk Gp \notag\\
&= (\cL_{Ax}  +\cL_{Gp}  +\cL_{Bu})\cL^{k}_{Ax}Gp
\end{align}
On the other hand,  it also holds that 
\begin{align}\label{eq:ddkgp2}
\ddtkp Gp &=\ddt \ddtk  Gz_2\notag \\
&=\ddt GM^kz_2(t) \notag \\
&=GM^k (Hz_1+Mz_2+Cu+E)
\end{align}
where we used the definition of the super-linearized dynamics.
Equating~\eqref{eq:ddkgp1} and~\eqref{eq:ddkgp2} at $t=0$, we obtain
\begin{multline}\label{eq:gencaseequatin}
\ddtkp|_{t=0} Gp(x)= 	(\cL_{Ax}  +\cL_{Gp}  +\cL_{Bu})\cL^{k}_{Ax}Gp\\=GM^k (Hx+Mp(x)+Cu+E).
\end{multline}
Equating the terms of same degree in $x$ and $u$ in~\eqref{eq:gencaseequatin},   we get  
\begin{equation}\label{eq:morezeros}\cL_{Bu}\cL^{k}_{Ax}Gp=0\mbox{ and } GM^kH=GM^kC=GM^kE=0.
\end{equation}	
 Using these relations in~\eqref{eq:ddkgp2}, we obtain
\begin{equation}\label{eq:dd2gpsimp}
\frac{d^{k+1}}{dt^{k+1}}Gp= 	GM^{k+1}p,
\end{equation}
proving the first part of item~\iii for $k+1$. Further  using them in~\eqref{eq:ddkgp1}, we obtain
\begin{equation}\label{eq:dd2gpsimp2}
\frac{d^{k+1}}{dt^{k+1}}Gp= 	\cL_{Ax+Gp}\cL^{k}_{Ax}Gp.
\end{equation}
Using  item~\iii of the induction hypothesis and~\eqref{eq:dd2gpsimp} then yield
\begin{equation}
	\frac{d^{k+1}}{dt^{k+1}} q(x(t)) =\frac{d}{dt}\frac{d^{k}}{dt^{k}} q(x(t)) =\frac{d}{dt} \cL^{k}_{Ax}q =  e_1^\top M^{k+1}q(x(t)).
\end{equation}

 We now evaluate the time derivative of $\cL^{k}_{Ax}q$
\begin{multline}\label{eq:ddta11l}
\ddt \cL^k_{Ax}q = \cL_{Ax+ Gp+Bu}\cL^{k}_{Ax}q= \left(\cL^{k+1}_{Ax}+\cL_{ Gp}\cL^{k}_{Ax}+\cL_{B u}\cL^{k}_{Ax}\right)q\\= e_1^\top M^{k+1}p(x(t)).
\end{multline}
We focus on the last equation in~\eqref{eq:ddta11l}. Evaluating it at $t=0$, we see that the only term containing the control is $u \cL_{B}\cL^k_{Ax}q $ on the left-hand side, and no term contains the control on the right-hand side. We conclude from it that 
\begin{equation}\label{eq:item2j0}
\cL_{B}\cL^k_{Ax}q =0.
\end{equation}
This proves item~\ii  of the induction hypothesis for $j=0$ and $k+1$.

Expanding on this equation, we have
\begin{align*} 
0=\cL_{B}\cL^k_{Ax}q&=\cL_{B}\cL_{Ax}\cL^{k-1}_{Ax}q
\end{align*}Using~\eqref{eq:propliebracket}, we get
$$\cL_{B}\cL_{Ax} \cL_{Ax}^{k-1}q= (\cL_{[B,Ax]}-\cL_{Ax}\cL_{B})\cL_{Ax}^{k-1}q.
$$
The second term on the right-hand side vanishes by item~\ii for $j=0$ and $k$ of the induction hypothesis.  The Lie bracket above is easily evaluated to be $[ B,Ax]=A B.$
We have thus shown that 
\begin{equation}\label{eq:adb1ada11l}0=\cL_{B}\cL_{Ax}^{k}q = \cL_{AB}\cL_{Ax}^{k-1}q, 
\end{equation}
which proves item~\ii of the inductive step for $j=1$ and $k+1$. Applying the steps above, namely factoring $\cL_{Ax}^{k-1}$ as $\cL_{Ax}\cL_{Ax}^{k-2}$ and using~\eqref{eq:propliebracket}, we obtain

\begin{align*}
0=	\cL_{AB}\cL_{Ax}^{k-1}q &=\cL_{AB}\cL_{Ax}\cL_{Ax}^{k-2}q\\
&= (\cL_{[AB,Ax]}-\cL_{Ax}\cL_{AB})\cL_{Ax}^{k-2}q\\
&= (\cL_{A^2B}-\cL_{Ax}\cL_{AB})\cL_{Ax}^{k-2}q.
\end{align*}
The second term in the last expression vanishes by the item~\ii  of the induction hypothesis for $j=1$ and $k-1$, which shows that $\cL_{A^2B}\cL_{Ax}^{k-2}q=0$ and proves the second item of the induction hypothesis for $j=2$ and $k+1$.

Iterating the steps above, we get in general
$$
0 = (\cL_{A^jB}\cL_{Ax}^{k-j}-\cL_{Ax}\cL_{A^{j-1}B} \cL_{Ax}^{k-j})q \mbox{ for }j \in \{1,\ldots,k \}
$$
The second term vanishes by item~\ii of the induction hypothesis with $j-1$ and $k$, and this proves item~\ii for $j$ and $k+1$.
We can proceed in this fashion until $j=k$, then the use of the induction hypothesis for $j=k-1$ and $k-1$, proves that $\cL_B\cL^k_{Ax}q=\cL_{A^kB} q =0$, which in turns proves item~\ii for $j=k$ and $k+1$. This conclude the proof of item~\ii.

We now focus on the proof of the first item. From~\eqref{eq:ddta11l} evaluated at $t=0$ and item~\ii, we have that for all $x \in \R^n$ 
\begin{align*}
\cL_{\bar G q}\cL^{k}_{Ax}q(x)=e_1^\top M^{k+1}p(x)-\cL^{k+1}_{Ax}q(x).
\end{align*}

Because the super-linearization is balanced, the right-hand side is a polynomial with terms of degree no higher than the degree of $q$ and thus no higher than the degree of $\cL^{k}_{Ax}q(x)$. Hence, by Lemma~\ref{lem:dpsiA12}, it holds that 
\begin{equation}\label{eq:item1j0}
\cL_{\bar Gq}\cL^{k}_{Ax}q(x)=\cL_{\bar G}\cL^{k}_{Ax}q(x)=0,
\end{equation}
proving item~1 of the induction hypothesis for $j=0$ and $k+1$.

Using the same approach as in the proof of item~\ii, we obtain from~\eqref{eq:item1j0}
\begin{align*}
0=\cL_{\bar G}\cL^{k}_{Ax}q(x) &=\cL_{\bar G}\cL_{Ax}\cL^{k-1}_{Ax}q(x)\\
&= \left(\cL_{[\bar G,Ax]}-\cL_{Ax} \cL_{\bar G} \cL^{k-1}_{Ax}\right) q(x)\\
&=\left(\cL_{A\bar G}-\cL_{Ax} \cL_{\bar G} \cL^{k-1}_{Ax}\right) q(x),
\end{align*}
Using item~\i of the induction hypothesis for $j=0$ and $k$, the last term in the previous equation is seen to be zero and we have that
$$
\cL_{A\bar G}\cL^{k-1}_{Ax} q=0,
$$
thus proving item~\i of the induction hypothesis for $j=1$ and $k+1$. The proof now proceeds along the exact lines of the proof of the second item, with $\bar G$ instead of $B$.

It remains to prove the last part of item~3 of the induction hypothesis for $k+1$. We have from~\eqref{eq:ddta11l}
\begin{align*}
\frac{d^{k+1}}{dt^{k+1}}q &= \left(\cL^{k+1}_{Ax}+\cL_{\bar Gq}\cL^{k}_{Ax}+\cL_{B u}\cL^{k}_{Ax}\right)q\\
&=  \cL^{k+1}_{Ax}q,
\end{align*}
where we used~\eqref{eq:item2j0} and~\eqref{eq:item1j0} to obtain the second equality. This concludes the proof.
 \end{proof}

We are now in a position to prove Theorem~\ref{th:canon}.
\begin{proof}[Proof of Theorem~\ref{th:canon}]
Let $V \subset \R^n$ be the {\em largest} linear subspace so that $$V \subseteq \bigcap_{x \in \R^n} \ker G\frac{\partial p}{\partial x}|_x.$$

Let $k:=\dim V$. From Proposition~\ref{prop:mainprop}, we have that  $\bar G \in \ker G\frac{\partial p}{\partial x}$ for all $x \in \R^n$ and thus $k>0$.  Let $T \in GL(n)$ be so that the first $k$ columns of $T$ span the vector subspace $V$, and the remaining columns are arbitrary. Introduce the variable $x'$, defined by 
$$
Tx'= x
$$
and partition them as $$x' :=\begin{bmatrix}  \bar x_1'\\  \bar x_2'\end{bmatrix}$$ where $ \bar x'_1 \in \R^k$ and $\bar x'_2 \in \R^{n-k}.$ We will denote by $[x']_j$ the $j$th entry of the vector $x'$.

Using Theorem~\ref{th:recmain1cor2}, we have that the dynamics~\eqref{eq:mainsys} is, expressed in the $x'$ variables, equivalent to
\begin{align}\label{eq:dynx'}
	\dot x'&= T^{-1}ATx' + T^{-1}Gp(Tx') + T^{-1}Bu\\
	&=T^{-1}ATx' + T^{-1}\bar G q(Tx') + T^{-1}Bu\\
	&= A'x'+ \bar G' q'(x')+ B' u,
\end{align}
where the last row defines $A', \bar G', q'(x')$ and $B'$.

By Proposition~\ref{prop:mainprop}, we have that  
$$\Span \{\bar G, A\bar G,\ldots\} \subseteq V$$ and $$\Span \{B, AB,\ldots\} \subseteq V.$$
It is then easy to see from the definition of $T$ that $[\bar G']_{k+1\cdots n}=0$,  $[B']_{k+1\cdots n}=0$ and that $A'$ has the block decomposition described in~\eqref{eq:canonicalformexplicit}. 

Finally, we have that 
$$
\frac{\partial q'(x')}{\partial [x']_j}=\frac{\partial q(x)}{\partial x}|_{x=Tx'}\frac{(\partial Tx')}{\partial [x']_j} =\frac{\partial q(x)}{\partial x}|_{x=Tx'}T_j 
$$
where $T_j$ is the $j$th column of $T$. From the definition of $T$, we see  that $V = \Span\{T_1,\ldots,T_k\}$. Thus, by definition of $V$ and Proposition~\ref{prop:mainprop}, it holds that
$$\frac{\partial q'(x')}{\partial x'_j}=\frac{\partial q}{\partial x}|_x T_j=0 \mbox{ for all }x \in \R^n \mbox{ and }1 \leq j \leq k,
$$ which implies that $q'(x')=q'(\bar x'_2)$. This concludes the proof. 
\end{proof}

\section{Summary and outlook}

Understanding which type of dynamics admit an embedding as a finite-dimensional affine system is an open problem with important ramifications. Indeed, such an embedding allows one to leverage methods from linear systems theory to control non-linear systems. This raises the equally important question of  whether systems admitting such an embedding, also called super-linearization, have a dynamics which is complex enough to warrant the effort of looking for one.  In this paper, we have taken the first step to address these issues by deriving a canonical form for some super-linearizable systems. 

Addressing the issue of super-linearization in general is a difficult task, and one has to understand first what characteristics of  super-linearizations are key to obtaining an amenable classification of super-linearizable systems. Based on our earlier work~\cite{belabbasobs2022}, we propose their number of visible observables as such a characteristic and, consequently,  have introduced here the class of {\em balanced} super-linearization with a unique visible observable. Note that this class does not restrict the total number of added observables. We have then shown that this class of systems admits a rather strong canonical form, which in essence shows that the nonlinear part of the dynamics depends on variables evolving autonomously and linearly. We will extend this approach to a larger class of system in future work.

\bibliographystyle{plain}
\bibliography{carleref.bib}
\end{document}